\UseRawInputEncoding

\documentclass[oneside, 11pt]{amsart} 

\usepackage{amsmath,amsthm,amssymb,epic}
\usepackage{eqlist,eqparbox}
\usepackage{amsfonts}
\usepackage{latexsym}
\usepackage[2emode]{psfrag}
\usepackage{amsthm}
\usepackage{amsmath}
\usepackage[all]{xy}

\newcommand{\Title}[1]{\bigskip\bigskip\centerline{\bf #1}\bigskip}
\newcommand{\Author}[1]{\medskip\centerline{ \it #1}}

\newcommand{\Affiliation}[1]{\medskip\centerline{#1}}
\newcommand{\Email}[1]{\medskip\centerline{#1}\bigskip}

\begin{document}

\newcommand{\N}{\mbox {$\mathbb N $}}
\newcommand{\Z}{\mbox {$\mathbb Z $}}
\newcommand{\Q}{\mbox {$\mathbb Q $}}
\newcommand{\R}{\mbox {$\mathbb R $}}
\newcommand{\lo }{\longrightarrow }
\newcommand{\ul}{\underleftarrow }
\newcommand{\rl}{\underrightarrow }
\newcommand{\rs }{\rightsquigarrow }
\newcommand{\ra }{\rightarrow }
\newcommand{\dd }{\rightsquigarrow }
\newcommand{\ol }{\overline }
\newcommand{\la }{\langle }
\newcommand{\tr }{\triangle }
\newcommand{\xr }{\xrightarrow }
\newcommand{\de }{\delta }
\newcommand{\pa }{\partial }
\newcommand{\LR }{\Longleftrightarrow }
\newcommand{\Ri }{\Rightarrow }
\newcommand{\va }{\varphi }
\newcommand{\Den}{{\rm Den}\,}
\newcommand{\Ker}{{\rm Ker}\,}
\newcommand{\Reg}{{\rm Reg}\,}
\newcommand{\Fix}{{\rm Fix}\,}
\newcommand{\Sup}{{\rm Sup}\,}
\newcommand{\Inf}{{\rm Inf}\,}
\newcommand{\Img}{{\rm Im}\,}
\newcommand{\Id}{{\rm Id}\,}
\newcommand{\ord}{{\rm ord}\,}

\newtheorem{theorem}{Theorem}[section]
\newtheorem{lemma}[theorem]{Lemma}
\newtheorem{proposition}[theorem]{Proposition}
\newtheorem{corollary}[theorem]{Corollary}
\newtheorem{definition}[theorem]{Definition}
\newtheorem{example}[theorem]{Example}
\newtheorem{examples}[theorem]{Examples}
\newtheorem{xca}[theorem]{Exercise}
\theoremstyle{remark}
\newtheorem{remark}[theorem]{Remark}
\newtheorem{remarks}[theorem]{Remarks}
\numberwithin{equation}{section}

\def\leftmark{L.C. Ciungu}

\Title{QUOTIENT QUANTUM-WAJSBERG ALGEBRAS} 
\title[Quotient quantum-Wajsberg algebras]{}
                                                                           
\Author{\textbf{LAVINIA CORINA CIUNGU}}
\Affiliation{Department of Mathematics} 
\Affiliation{St Francis College}
\Affiliation{180 Remsen Street, Brooklyn Heights, NY 11201-4398, USA}
\Email{lciungu@sfc.edu}

\begin{abstract} 
We define and study the notions of q-deductive systems, p-deductive systems, deductive systems, maximal and strongly maximal q-deductive systems in quantum-Wajsberg algebras.  
We also introduce the notion of congruences induced by deductive systems of a quantum-Wajsberg algebra, 
and we show that there is a relationship between congruences and deductive systems. 
Furthermore, we define the quotient quantum-Wajsberg algebra with respect to a deductive system, and prove that the 
quotient quantum-Wajsberg algebra is locally finite if and only if the deductive system is strongly maximal. 
Finally, we define the weakly linear and quasi-linear quantum-Wajsberg algebras, and give a characterization of 
weakly linear quotient quantum-Wajsberg algebras. \\ 

\textbf{AMS classification (2020):} 06F35, 03G25, 06A06, 81P10
\end{abstract}

\maketitle

\section{Introduction} 

As G. Birkhoff and J. von Neumann showed in their paper ``The logic of quantum mechanics" (\cite{Birk1}), 
the set of assertions of quantum mechanics has different algebraic properties from a Boolean algebra.  
In the last decades, developing algebras related to the logical foundations of quantum mechanics became a 
central topic of research. 
Generally known as quantum structures, these algebras serve as models for the formalism of quantum mechanics. 
R. Giuntini introduced in \cite{Giunt1} the quantum-MV algebras as non-lattice generalizations of 
MV algebras (\cite{Chang, Cig1}), and as non-idempotent generalizations of orthomodular lattices (\cite{Kalm, Ptak}).   
These structures were intensively studied by R. Giuntini (\cite{Giunt2, Giunt3, Giunt4, Giunt5, Giunt6}), 
A. Dvure\v censkij and S. Pulmannov\'a (\cite{DvPu}), R. Giuntini and S. Pulmannov\'a (\cite{Giunt7}) and by 
A. Iorgulescu in \cite{Ior30, Ior31, Ior32, Ior33, Ior34, Ior35}. \\
\indent
Quantum-B algebras defined and investigated by W. Rump and Y.C. Yang (\cite{Rump2, Rump1}) 
arise from the concept of quantales, which was introduced in 1984 as a framework for quantum mechanics 
with a view toward non-commutative logic (\cite{Mulv1}). 
Many implicational algebras studied so far (effect algebras, residuated lattices, MV/BL/MTL algebras, bounded R$\ell$-monoids, hoops, BCK/BCI algebras), as well as their non-commutative versions, are quantum-B algebras. Interesting results on quantum-B algebras have been presented in \cite{Rump3, Rump4, Han1, Han2}. \\
\indent
We redefined in \cite{Ciu78} the quantum-MV algebras starting from involutive BE algebras, and we introduced and 
studied the notion of quantum-Wajsberg algebras. 
We proved that the quantum-Wajsberg algebras are equivalent to quantum-MV algebras and that Wajsberg algebras are 
both quantum-Wajsberg algebras and commutative quantum-B algebras. \\
\indent
In this paper, we define and study the notions of q-deductive systems, p-deductive systems, deductive systems, maximal  and strongly maximal q-deductive systems in quantum-Wajsberg algebras, and prove that any strongly maximal q-deductive system is maximal. If every q-deductive system is a deductive system, we show that the notions of maximal and strongly maximal q-deductive systems coincide. 
We also introduce the notion of congruences induced by deductive systems of a quantum-Wajsberg algebra, 
and prove that there is a relationship between congruences and deductive systems. 
Furthermore, we define the quotient quantum-Wajsberg algebra with respect to a deductive system, and show that the 
quotient quantum-Wajsberg is locally finite if and only if the deductive system is strongly maximal. 
We define the weakly linear quantum-Wajsberg algebras, we investigate their properties, and give a characterization of a weakly linear quotient quantum-Wajsberg algebra.
Finally, we define and characterize the notion of quasi-linear quantum-Wajsberg algebras.

$\vspace*{5mm}$

\section{Preliminaries}

In this section, we recall some basic notions and results regarding BCK algebras, Wajsberg algebras, BE algebras and quantum-Wajsberg algebras that will be used in the paper. Additionally, we prove new properties of quantum-Wajsberg algebras. For more details regarding the quantum-Wajsberg algebras we refer the reader to \cite{Ciu78}. \\
\indent
Starting from the systems of positive implicational calculus, weak systems of positive implicational calculus 
and BCI and BCK systems, in 1966 Y. Imai and K. Is\`eki introduced the \emph{BCK algebras} (\cite{Imai}). 
BCK algebras are also used in a dual form, with an implication $\ra$ and with one constant element $1$, 
that is the greatest element (\cite{Kim2}). 
A (dual) \emph{BCK algebra} is an algebra $(X,\ra,1)$ of type $(2,0)$ satisfying the following conditions, 
for all $x,y,z\in X$: \\
$(BCK_1)$ $(x\ra y)\ra ((y\ra z)\ra (x\ra z))=1;$ \\
$(BCK_2)$ $1\ra x=x;$ \\
$(BCK_3)$ $x\ra 1=1;$ \\
$(BCK_4)$ $x\ra y=1$ and $y\ra x=1$ imply $x=y$. \\
In this paper, we use the dual BCK algebras. 
If $(X,\ra,1)$ is a BCK algebra, for $x,y\in X$ we define the relation $\le$ by $x\le y$ if and only if $x\ra y=1$, 
and $\le$ is a partial order on $X$. \\
\indent
\emph{Wajsberg algebras} were introduced in 1984 by Font, Rodriguez and Torrens in \cite{Font1} as algebraic model 
of $\aleph_0$-valued \L ukasiewicz logic.   
A \emph{Wajsberg algebra} is an algebra $(X,\ra,^*,1)$ of type $(2,1,0)$ satisfying the following conditions 
for all $x,y,z\in X$: \\
$(W_1)$ $1\ra x=x;$ \\
$(W_2)$ $(y\ra z)\ra ((z\ra x)\ra (y\ra x))=1;$ \\
$(W_3)$ $(x\ra y)\ra y=(y\ra x)\ra x;$ \\
$(W_4)$ $(x^*\ra y^*)\ra (y\ra x)=1$. \\
Wajsberg algebras are bounded with $0=1^*$, and they are involutive. 
It was proved in \cite{Font1} that Wajsberg algebras are termwise equivalent to MV algebras. \\
\indent
\emph{BE algebras} were introduced in \cite{Kim1} as algebras $(X,\ra,1)$ of type $(2,0)$ satisfying the 
following conditions, for all $x,y,z\in X$: \\
$(BE_1)$ $x\ra x=1;$ \\
$(BE_2)$ $x\ra 1=1;$ \\
$(BE_3)$ $1\ra x=x;$ \\
$(BE_4)$ $x\ra (y\ra z)=y\ra (x\ra z)$. \\
A relation $\le$ is defined on $X$ by $x\le y$ iff $x\ra y=1$. 
A BE algebra $X$ is \emph{bounded} if there exists $0\in X$ such that $0\le x$, for all $x\in X$. 
In a bounded BE algebra $(X,\ra,0,1)$ we define $x^*=x\ra 0$, for all $x\in X$. 
A bounded BE algebra $X$ is called \emph{involutive} if $x^{**}=x$, for any $x\in X$. \\
A BE algebra $X$ is called \emph{commutative} if $(x\ra y)\ra y=(y\ra x)\ra x$, for all $x,y\in X$. 
A bounded BE algebra $X$ is called \emph{involutive} if $x^{**}=x$, for any $x\in X$. \\
Obviously, any (left-)BCK algebra is a (left-)BE algebra, but the exact connection between BE algebras and 
BCK algebras is made in the papers \cite{Ior16, Ior17}: a BCK algebra is a BE algebra satisfying $(BCK_4)$ (antisymmetry) and $(BCK_1)$. 
 
\begin{lemma} \label{qbe-10} Let $(X,\ra,1)$ be a BE algebra. Then the following hold for all $x,y,z\in X$: \\
$(1)$ $x\ra (y\ra x)=1;$ $(2)$ $x\le (x\ra y)\ra y$. \\
If $X$ is bounded, then: \\
$(3)$ $x\ra y^*=y\ra x^*;$ \\
$(4)$ $x\le x^{**}$. \\
If $X$ is involutive, then: \\
$(5)$ $x^*\ra y=y^*\ra x;$ \\
$(6)$ $x^*\ra y^*=y\ra x;$ \\
$(7)$ $(x^*\ra y)^*\ra z=x^*\ra (y^*\ra z);$ \\
$(8)$ $x\ra (y\ra z)=(x\ra y^*)^*\ra z$.   
\end{lemma}
\begin{proof}
$(1)$ Using $(BE_4)$, we have $1=y\ra 1=y\ra (x\ra x)=x\ra (y\ra x)$. \\
$(2)$ By $(BE_4)$, we have $x\ra ((x\ra y)\ra y)=1$, that is $x\le (x\ra y)\ra y$. \\
$(3)$ It follows from $(BE_4)$ for $z:=0$. \\
$(4)$ It follows by $(2)$ for $y:=0$. \\
$(5)$ Replace $x$ by $x^*$ and $y$ by $y^*$ in $(3)$. \\
$(6)$ Replace $y$ by $y^*$ in $(5)$. \\
$(7)$ Using $(5)$ and $(BE_4)$, we get: \\
$\hspace*{2cm}$ 
$(x^*\ra y)^*\ra z=z^*\ra (x^*\ra y)=x^*\ra (z^*\ra y)=x^*\ra (y^*\ra z)$. \\
$(8)$ Using $(BE_4)$, we have: \\ 
$\hspace*{2cm}$ $x\ra (y\ra z)=x\ra (z^*\ra y^*)=z^*\ra (x\ra y^*)=(x\ra y^*)^*\ra z$. 
\end{proof}

\noindent
In a BE algebra $X$, we define the additional operation: \\
$\hspace*{3cm}$ $x\Cup y=(x\ra y)\ra y$. \\
If $X$ is involutive, we define the operations: \\
$\hspace*{3cm}$ $x\Cap y=((x^*\ra y^*)\ra y^*)^*$, \\
$\hspace*{3cm}$ $x\odot y=(x\ra y^*)^*=(y\ra x^*)^*$, \\
and the relation $\le_Q$ by: \\
$\hspace*{3cm}$ $x\le_Q y$ iff $x=x\Cap y$. \\
Note that in algebras with implication, as BCK algebras, the join operation is usually denoted by $\vee$ and the
meet operation by $\wedge$. \\

\begin{proposition} \label{qbe-20} $\rm($\cite{Ciu78}$\rm)$ Let $X$ be an involutive BE algebra. 
Then the following hold for all $x,y,z\in X$: \\
$(1)$ $x\le_Q y$ implies $x=y\Cap x$ and $y=x\Cup y;$ \\
$(2)$ $\le_Q$ is reflexive and antisymmetric; \\
$(3)$ $x\Cap y=(x^*\Cup y^*)^*$ and $x\Cup y=(x^*\Cap y^*)^*;$ \\ 
$(4)$ $x\le_Q y$ implies $x\le y;$ \\
$(5)$ $0\le_Q x \le_Q 1;$ \\
$(6)$ $0\Cap x=x\Cap 0=0$ and $1\Cap x=x\Cap 1=x;$ \\
$(7)$ $x\Cap (y\Cap x)=y\Cap x$ and $x\Cap (x\Cap y)=x\Cap y;$ \\
$(8)$ $(x\Cap y)\ra z=(y\ra x)\ra (y\ra z);$ \\
$(9)$ $z\ra (x\Cup y)=(x\ra y)\ra (z\ra y);$ \\
$(10)$ $x\Cap y\le x,y\le x\Cup y$. 
\end{proposition}

\begin{proposition} \label{qbe-20-10} Let $X$ be an involutive BE algebra. 
Then the following hold for all $x,y,z\in X$: \\
$(1)$ $x, y\le_Q z$ and $z\ra x=z\ra y$ imply $x=y$ \emph{(cancellation law)}; \\  
$(2)$ $x\le_Q y$ implies $(y\ra x)\odot y=x;$ \\ 
$(3)$ $x\ra (z\odot y^*)=((z\ra y)\odot x)^*$. 
\end{proposition}
\begin{proof}
$(1)$ Since $x,y\le_Q z$ and $z\ra x=z\ra y$, we have: \\
$\hspace*{2.00cm}$ $x=x\Cap z=((x^*\ra z^*)\ra z^*)^*=((z\ra x)\ra z^*)^*$ \\
$\hspace*{2.30cm}$ $=((z\ra y)\ra z^*)^*=((y^*\ra z^*)\ra z^*)^*=y\Cap z=y$. \\
$(2)$ Since $x\le_Q y$, we have: \\
$\hspace*{2cm}$ $(y\ra x)\odot y=((y\ra x)\ra y^*)^*=((x^*\ra y^*)\ra y^*)^*=x\Cap y=x$. \\
$(3)$ Taking into consideration that $x\ra y=(x\odot y^*)^*$, we get: \\
$\hspace*{2cm}$ $x\ra (z\odot y^*)=(x\odot (z\odot y^*)^*)^*=(x\odot (z\ra y))^*$. 
\end{proof}

A \emph{(left-)quantum-Wajsberg algebra} (\emph{QW algebra, for short}) $(X,\ra,^*,1)$ is an involutive BE algebra $(X,\ra,^*,1)$ satisfying the following condition for all $x,y,z\in X$: \\
(QW) $x\ra ((x\Cap y)\Cap (z\Cap x))=(x\ra y)\Cap (x\ra z)$. 

\noindent
Condition (QW) is equivalent to the following conditions: \\
($QW_1$) $x\ra (x\Cap y)=x\ra y;$ \\ 
($QW_2$) $x\ra (y\Cap (z\Cap x))=(x\ra y)\Cap (x\ra z)$. 

\begin{definition} \label{qmv-30} $\rm($\cite{Ior30}$\rm)$
\emph{      
A \emph{(left-)m-BE algebra} is an algebra $(X,\odot,^{*},1)$ of type $(2,1,0)$ satisfying the following properties, 
for all $x,y,z\in X$: \\ 
(PU) $1\odot x=x=x\odot 1;$ \\
(Pcomm) $x\odot y=y\odot x;$ \\
(Pass) $x\odot (y\odot z)=(x\odot y)\odot z;$ \\ 
(m-L) $x\odot 0=0;$ \\ 
(m-Re) $x\odot x^{*}=0$, \\
where $0:=1^*$. 
}\end{definition}
Note that, according to \cite[Cor. 17.1.3]{Ior35}, the involutive (left-)BE algebras $(X,\ra,^*,1)$ are definitionally equivalent to involutive (left-)m-BE algebras $(X,\odot,^*,1)$, by the mutually inverse transformations 
(\cite{Ior30, Ior35}): \\ 
$\hspace*{3cm}$ $\Phi:$\hspace*{0.2cm}$ x\odot y:=(x\ra y^*)^*$ $\hspace*{0.1cm}$ and  
                $\hspace*{0.1cm}$ $\Psi:$\hspace*{0.2cm}$ x\ra y:=(x\odot y^*)^*$. 
                
\begin{definition} \label{qmv-40} $\rm($\cite[Def. 3.10]{Ior34}$\rm)$
\emph{      
A \emph{(left-)quantum-MV algebra}, or a \emph{(left-)QMV algebra} for short, is an involutive (left-)m-BE algebra
$(X,\odot,^{*},1)$ verifying the following axiom: for all $x,y,z\in X$, \\
(Pqmv) $x\odot ((x^*\Cup y)\Cup (z\Cup x^*))=(x\odot y)\Cup (x\odot z)$. 
}
\end{definition}

\begin{proposition} \label{qmv-50}
The (left-)quantum-Wajsberg algebras are defitionally equivalent to (left-)quantum-MV algebras. 
\end{proposition}
\begin{proof}
We prove that the axioms $(Pqmv)$ and $(QW)$ are equivalent. 
Using the transformation $\Phi$, from (Pqmv) we get: \\
$\hspace*{0.5cm}$ $x\odot ((x^*\Cup y)\Cup (z\Cup x^*))=(x\ra ((x^*\Cup y)\Cup (z\Cup x^*))^*)^*
=(x\ra ((x\Cap y^*)\Cap (z^*\Cap x)))^*$ and \\
$\hspace*{0.5cm}$ $(x\odot y)\Cup (x\odot z)=(x\ra y^*)^*\Cup (x\ra z^*)^*=((x\ra y^*)\Cap (x\ra z^*))^*$, \\
hence (Pqmv) becomes: \\
$\hspace*{0.5cm}$ $(x\ra ((x\Cap y^*)\Cap (z^*\Cap x)))^*=((x\ra y^*)\Cap (x\ra z^*))^*$, \\
for all $x,y,z\in X$. Replacing $y$ by $y^*$ and $z$ by $z^*$, we get axiom (QW). 
Similarly axiom (QW) implies axiom (Pqmv). 
\end{proof}

In what follows, by quantum-MV algebras and quantum-Wajsberg algebras we understand the left-quantum-MV algebras and  left-quantum-Wajsberg algebras, respectively. 

\begin{proposition} \label{qbe-60} $\rm($\cite{Ciu78}$\rm)$ Let $(X,\ra,^*,1)$ be a quantum-Wajsberg algebra. 
The following hold for all $x,y\in X$:\\
$(1)$ $x\ra (y\Cap x)=x\ra y$ and $(x\ra y)\ra (y\Cap x)=x;$ \\
$(2)$ $x\le_Q x^*\ra y$ and $x\le_Q y\ra x;$ \\
$(3)$ $x\ra y=0$ iff $x=1$ and $y=0;$ \\
$(4)$ $(x\ra y)^*\Cap x=(x\ra y)^*;$ \\ 
$(5)$ $(x\Cap y)\Cap y=x\Cap y$ and $(x\Cup y)\Cup y=x\Cup y;$ \\
$(6)$ $x\Cup (y\Cap x)=x$ and $x\Cap (y\Cup x)=x;$ \\
$(7)$ $x\Cap y\le_Q y\le_Q x\Cup y;$ \\
$(8)$ $(x\Cup y)\ra x=(y\Cup x)\ra x=y\ra x;$ \\
$(9)$ $(x\Cup y)\ra y=(y\Cup x)\ra y=x\ra y;$ \\
$(10)$ $x\le y$ iff $y\Cap x=x$.      
\end{proposition}

\begin{proposition} \label{qbe-70} $\rm($\cite{Ciu78}$\rm)$ Let $(X,\ra,^*,1)$ be a quantum-Wajsberg algebra. 
If $x,y\in X$ such that $x\le_Q y$, then the following hold for any $z\in X$:\\
$(1)$ $y=y\Cup x;$ \\
$(2)$ $y^*\le_Q x^*;$ \\
$(3)$ $y\ra z\le_Q x\ra z$ and $z\ra x\le_Q z\ra y;$ \\
$(4)$ $x\Cap z\le_Q y\Cap z$ and $x\Cup z\le_Q y\Cup z$. 
\end{proposition}

\begin{proposition} \label{qbe-80} $\rm($\cite{Ciu78}$\rm)$ Let $(X,\ra,^*,1)$ be a quantum-Wajsberg algebra. 
The following hold, for all $x,y,z\in X$:\\
$(1)$ $(x\Cap y)\Cap (y\Cap z)=(x\Cap y)\Cap z;$ \\
$(2)$ $\le_Q$ is transitive; \\
$(3)$ $x\Cup y\le_Q x^*\ra y;$ \\
$(4)$ $(x^*\ra y)^*\ra (x\ra y^*)^*=x^*\ra y;$ \\
$(5)$ $(x\ra y)^*\ra (y\ra x)^*=x\ra y;$ \\
$(6)$ $(y\ra x)\ra (x\ra y)=x\ra y;$ \\
$(7)$ $(x\ra y)\Cup (y\ra x)=1;$ \\
$(8)$ $(z\Cap x)\ra (y\Cap x)=(z\Cap x)\ra y;$ \\
$(9)$ $(x\Cap y)\ra (y\Cap x)=1;$ \\
$(10)$ $(x\Cup y)\ra (y\Cup x)=1$.     
\end{proposition}

\noindent
By Propositions \ref{qbe-20}$(2)$, \ref{qbe-80}$(2)$, in a quantum-Wajsberg algebra $X$, $\le_Q$ is a partial order on $X$. 

\begin{proposition} \label{qbe-100} Let $(X,\ra,^*,1)$ be a quantum-Wajsberg algebra. 
Then the following hold, for all $x,y\in X$:\\
$(1)$ $(x\Cap y)^*\ra (y\Cap x)^*=1$ and $(x\Cap y)\ra (y\Cap x)=1;$ \\
$(2)$ $(x\Cup y)^*\ra (y\Cup x)^*=1$ and $(x\Cup y)\ra (y\Cup x)=1;$ \\ 
$(3)$ $x\Cap y=0$ iff $y\Cap x=0;$ \\
$(4)$ $x\Cup y=1$ iff $y\Cup x=1$. 
\end{proposition}
\begin{proof}
$(1)$ Using $(QW_1)$ we have: \\
$\hspace*{1.00cm}$ $(x\Cap y)^*\ra (y\Cap x)^*=(x\Cap y)^*\ra ((y^*\ra x^*)\ra x^*)
                    =(y^*\ra x^*)\ra ((x\Cap y)^*\ra x^*)$ \\
$\hspace*{4.30cm}$ $=(x\ra y)\ra (x\ra (x\Cap y))=(x\ra y)\ra (x\ra y)=1$. \\
It implies $(y\Cap x)\ra (x\Cap y)=1$, and changing $x$ and $y$ we get the second identity. \\
$(2)$ It follows from $(1)$ replacing $x$ by $x^*$ and $y$ by $y^*$. \\
$(3)$ If $x\Cap y=0$, then by $(1)$, $1\ra (y\Cap x)^*=1$, so that $(y\Cap x)^*=1$, that is $y\Cap x=0$. 
Conversely, if $y\Cap x=0$, by the second identity of $(1)$ we get $(x\Cap y)^*=1$, hence $x\Cap y=0$. \\
$(4)$ If $x\Cup y=1$, using the second identity of $(2)$ we get $y\Cup x=1$. 
Conversely, if $y\Cup x=1$, by the first identity of $(2)$ we get $x\Cup y=1$. 
\end{proof}

\begin{proposition} \label{qbe-120} In any quantum-Wajsberg algebra $(X,\ra,^*,1)$ the following hold 
for all $x,y,z\in X$: \\
$(1)$ $x\ra (y\ra z)=(x\odot y)\ra z;$ \\
$(2)$ $x\le_Q y\ra z$ implies $x\odot y\le z;$ \\
$(3)$ $x\odot y\le z$ implies $x\le y\ra z;$ \\
$(4)$ $(x\ra y)\odot x\le y;$ \\
$(5)$ $x\le_Q y$ implies $x\odot z\le_Q y\odot z$. 
\end{proposition}
\begin{proof}
$(1)$ Using Lemma \ref{qbe-10}, we have: \\
$\hspace*{1cm}$ $x\ra (y\ra z)=x\ra (z^*\ra x^*)=z^*\ra (x\ra y^*)=(x\ra y^*)^*\ra z=(x\odot y)\ra z$. \\
$(2)$ $x\le_Q y\ra z$ implies $x\le y\ra z$, so that $x\ra (y\ra z)=1$. 
Hence $(x\odot y)\ra z=1$, that is $x\odot y\le z$. \\
$(3)$ From $x\odot y\le z$ we have $(x\odot y)\ra z=1$, that is $(x\ra y^*)^*\ra z=1$. 
It follows that $z^*\ra (x\ra y^*)=1$, so that $x\ra (z^*\ra y^*)=1$, hence $x\ra (y\ra z)=1$ and so 
$x\le y\ra z$. \\
$(4)$ Since $x\ra y\le_Q x\ra y$, by $(2)$ we get $(x\ra y)\odot x\le y$. \\
$(5)$ From $x\le_Q y$ we get $y\ra z^*\le_Q x\ra z^*$ and $(x\ra z^*)^*\le_Q (y\ra z^*)^*$, that is 
$x\odot z\le_Q y\odot z$. 
\end{proof}

A quantum-Wajsberg algebra $X$ is called \emph{commutative} if $x\Cup y=y\Cup x$, for all $x,y\in X$. 
It was proved in \cite{Ciu78} that any Wajsberg algebra is a quantum-Wajsberg algebra, and a 
quantum-Wajsberg algebra is a Wajsberg algebra if and only if the relations $\le$ and $\le_Q$ coincide. 

\begin{remark} \label{qbe-130}
Since: \\
- commutative BE algebras are commutative BCK algebras (\cite{Walend1}]),  \\
- bounded commutative BCK are definitionally equivalent to MV algebras (\cite{Mund1}) and \\
- Wajsberg algebras are definitionally equivalent to MV algebras (\cite{Font1}), \\
it follows that bounded commutative BE algebras are bounded commutative BCK algebras, hence are definitionally equivalent with MV algebras, hence with Wajsberg algebras. 
Hence the commutative quantum-Wajsberg algebras are the Wajsberg algebras. 
\end{remark} 

$\vspace*{5mm}$

\section{Deductive systems in quantum-Wajsberg algebras}

The ideals in right-QMV algebras were defined by R. Giuntini and S. Pulmannov\'a in \cite{Giunt7} 
(see also \cite{DvPu, DvPu1}). They also introduced the notion of perspective elements in QMV algebras. 
Using properties of these elements, the authors studied the ideals of QMV algebras.  
In this section, we extend these notions to the case of QW algebras. 
We define the q-deductive systems, dual-perspective elements, p-deductive systems and deductive systems in QW 
algebras, and show that every deductive system is a p-deductive system. 
Following the ideas from the paper \cite{Bus1}, we prove that the q-deductive systems and p-deductive systems 
of the quantum-Wajsberg algebra $(X,\ra,^*,1)$ coincide with the q-filters and p-filters of the term equivalent quantum-MV algebra $(X,\odot,^*,1)$.  
We also define the maximal and strongly maximal q-deductive system, and prove that any strongly maximal 
q-deductive system is maximal. 
If every q-deductive system is a deductive system, we show that the notions of maximal and strongly maximal 
q-deductive systems coincide. \\
In what follows, $(X,\ra,^*,1)$ will be a quantum-Wajsberg algebra, unless otherwise stated. 

\begin{definition} \label{fqbe-10}
\emph{
We say that the elements $x,y\in X$ are in \emph{dual-perspective}, denoted by 
$x\sim y$, if there exists $\alpha\in X$ such that $x\le \alpha\le x$ and $y\le \alpha\le y$. 
}
\end{definition}

\begin{lemma} \label{fqbe-20} Let $X$ be a QW algebra. The following hold for all $x,y\in X$: \\
$(1)$ the relation $\sim$ is reflexive and symmetric; \\ 
$(2)$ $x\sim 0$ implies $x=0$ and $x\sim 1$ implies $x=1;$ \\
$(3)$ $x\ra y=y\ra x=1$ implies $x\sim y;$ \\
$(4)$ $x\sim y$ iff $x^*\sim y^*;$ \\
$(5)$ $(x\Cap y)\sim (y\Cap x)$ and $(x\Cup y)\sim (y\Cup x)$. 
\end{lemma}
\begin{proof}
$(2)$ If $x\sim 0$, then there exists $\alpha\in X$ such that $x\ra \alpha=\alpha\ra x=1$ and 
$0\ra \alpha=\alpha\ra 0=1$. It follows that $\alpha=0$, so that $x=0$. 
Similarly, if $x\sim 1$, then there exists $\alpha\in X$ such that $x\ra \alpha=\alpha\ra x=1$ and 
$1\ra \alpha=\alpha\ra 1=1$. We get $\alpha=1$ and $x=1$. \\
$(3)$ Taking $\alpha=y$, we have $x\le \alpha\le x$ and $y\le \alpha\le y$, that is $x\sim y$. \\
$(4)$ We have $x\sim y$, there exists $\alpha\in X$ such that $x\le \alpha\le x$ and $y\le \alpha\le y$.  
Hence $x^*\le \alpha^*\le x^*$ and $y^*\le \alpha^*\le y^*$, that is $x^*\sim y^*$. 
The converse follows similarly. \\
$(5)$ By Proposition \ref{qbe-100}, $(x\Cap y)\ra (y\Cap x)=(y\Cap x)\ra (x\Cap y)=1$, and taking 
$\alpha:=y\Cap x$ we get $(x\Cap y)\sim (y\Cap x)$. 
Similarly, from $(x\Cup y)\ra (y\Cup x)=(y\Cup x)\ra (x\Cup y)=1$ we have $(x\Cup y)\sim (y\Cup x)$. 
\end{proof}

\begin{definition} \label{fqbe-20-10}  
\emph{
Let $(X,\odot,^*,1)$ be a QMV algebra. A nonempty subset $F$ of $X$ is called a \emph{q-filter} of $X$, if it 
satisfies the following conditions:\\
$(pf_1)$ $x,y\in F$ implies $x\odot y\in F;$ \\
$(pf_2)$ $x\in F$, $y\in X$ imply $x\oplus y\in X$. \\
A q-filter is called a \emph{p-filter} if it satisfies the following condition: \\
$(pf_3)$ $x\in F$, $y\in X$ imply $x\Cup y\in X$. 
}
\end{definition}

Note that the q-filters and p-filters are duals of the q-ideals and p-ideals defined in \cite{Giunt7} for the 
case of right-QMV algebras.  

\begin{definition} \label{fqbe-30} 
\emph{
A \emph{q-deductive system} of $X$ is a nonempty subset $F\subseteq X$ satisfying the folowing conditions: \\
$(F_1)$ $x,y\in F$ implies $(x\ra y^*)^*\in F$ ($x\odot y\in F$); \\
$(F_2)$ $x\in F$, $y\in X$ imply $y\ra x\in F$. 
}
\end{definition}

\begin{proposition} \label{fqbe-40} 
A nonempty subset $F\subseteq X$ is a q-deductive system of $X$ if and only if it satisfies conditions $(F_1)$ and \\ 
$(F_2^{'})$ $x\in F$, $y\in X$, $x\le_Q y$ imply $y\in F$.   
\end{proposition}
\begin{proof}
We show that conditions $(F_2)$ and $(F_2^{'})$ are equivalent. \\
Assume that $F$ satisfies condition $(F_2)$, and let $x\in F$ and $y\in X$ such that $x\le_Q y$. 
According to Proposition \ref{qbe-70}$(1)$ and by $(F_2)$, $y=y\Cup x=(y\ra x)\ra x\in F$, that is 
condition $(F_2^{'})$ is verified.  
Conversely, if $x\in F$ and $y\in X$, by Proposition \ref{qbe-60}$(2)$, $x\le_Q y\ra x$, and by $(F_2^{'})$ 
we get $y\ra x\in F$. Hence condition $(F_2)$ is satisfied. 
\end{proof}

\begin{definition} \label{fqbe-40-10} 
\emph{
A q-deductive system of $X$ is called a \emph{p-deductive system} if it satisfies the following condition: \\
$(F_3)$ $x\in F$ and $y\in X$ implies $x\Cup y\in F$. 
}
\end{definition}

Denote by $\mathcal{DS}_q$, $\mathcal{DS}_p$ the set of all q-deductive system and p-deductive system of $X$, respectively.

\begin{remark} \label{fqbe-40-20} 
The q-deductive systems and p-deductive systems in a quantum-Wajsberg algebra $(X,\ra,*,1)$ coincide with the 
q-filters and p-filters in the quantum-MV algebra $(X,\odot,^*,1)$. \\
Indeed, let $F\in \mathcal{DS}_p$. 
Obviously, $(F_1)$ and $(F_3)$ coincide with $(pf_1)$ and $(pf_3)$, respectively. 
Suppose that $X$ satisfies $(F_2)$, and let $x\in F$, $y\in X$. 
By $(F_2)$ we get $y^*\ra x\in F$, that is $x\oplus y\in F$, so that $(pf_2)$ is satisfied. 
Conversely, assume that $X$ verifies $(pf_2)$, and consider $x\in F$, $y\in X$. 
Since $X=\{x^*\mid x\in X\}$, there exists $y_1\in X$ such that $y=y_1^*$. 
By $(pf_2)$, $y_1\oplus x\in F$, so that $y\ra x=y_1^*\ra x\in F$. 
Hence $X$ satisfies $(F_2)$.  
\end{remark}

\begin{definition} \label{fqbe-50} 
\emph{
A \emph{deductive system} of $X$ is a subset $F\subseteq X$ satisfying the following conditions: \\
$(DS_1)$ $1\in F;$ \\
$(DS_2)$ $x, x\ra y\in F$ implies $y\in F$. 
}
\end{definition}

Denote by $\mathcal{DS}(X)$ the set of all deductive systems of $X$. 
We say that $F\in \mathcal{DS}(X)$ is \emph{proper} if $F\neq X$. 
Obviously $\{1\},X\in \mathcal{DS}(X)$. 

\begin{remark} \label{fqbe-60} 
$(1)$ Taking into consideration $(QW_1)$ and Proposition \ref{qbe-60}$(1)$, condition $(DS_2)$ is equivalent 
to each of the following conditions: \\
$(DS_2^{'})$ $x,x\ra (x\Cap y)\in F$ implies $y\in F;$ \\
$(DS_2^{''})$ $x,x\ra (y\Cap x)\in F$ implies $y\in F$. \\
$(2)$ Given $F\in \mathcal{DS}(X)$, by Proposition \ref{qbe-80}$(9)$,$(10)$ we get: 
$x\Cap y\in F$ iff $y\Cap x\in F$ and $x\Cup y\in F$ iff $y\Cup x\in F$. 
\end{remark}

\begin{proposition} \label{fqbe-70} $\mathcal{DS}(X)\subseteq \mathcal{DS}_q(X)$. 
\end{proposition}
\begin{proof}
Let $F\in \mathcal{DS}(X)$ and let $x,y\in F$. Since by $(DS_1)$, $1\in F$, it follows that $F$ is nonempty. 
By Lemma \ref{qbe-10}$(8)$ we have $x\ra (y\ra z)=(x\ra y^*)^*\ra z$, for all $x,y,z\in X$. 
Taking $z:=(x\ra y^*)^*$, we get $x\ra (y\ra (x\ra y^*)^*)=(x\ra y^*)^*\ra (x\ra y^*)^*=1\in F$. 
Since $x,y\in F$, by $(DS_2)$ we get $(x\ra y^*)^*\in F$, that is $(F_1)$. \\ 
Let $x,y\in F$ such that $x\le_Q y$. By Lemma \ref{qbe-20}$(4)$, we have $x\le y$, that is $x\ra y=1\in F$. 
Since $x\in F$, by $(DS_2)$ we get $y\in F$, hence condition $(F_2^{'})$ is also satisfied. 
By Proposition \ref{fqbe-40}, it follows that $F\in \mathcal{DS}_q(X)$, hence 
$\mathcal{DS}(X)\subseteq \mathcal{DS}_q(X)$.  
\end{proof}

\begin{proposition} \label{fqbe-80} Let $F$ be a subset $F\subseteq X$. The following are equivalent: \\
$(a)$ $F\in \mathcal{DS}(X);$ \\
$(b)$ $F$ is nonempty and it satisfies conditions $(F_1)$ and \\
$(F_4)$ $x\in F$, $y\in X$, $x\le y$ imply $y\in F;$ \\
$(c)$ $F$ is nonempty and it satisfies conditions $(F_1)$ and $(F_3);$ \\
$(d)$ $F\in \mathcal{DS}_q(X)$ satisfying condition \\ 
$(F_5)$ $x\in F$, $y\in X$, $x\sim y$ imply $y\in F$. 
\end{proposition}
\begin{proof}
$(a)\Rightarrow (b)$ Since $F\in \mathcal{DS}(X)$, according to Proposition \ref{fqbe-70}, $F$ satisfies 
condition $(F_1)$. 
Let $x\in F$, $y\in X$ such that $x\le y$. 
It follows that $x\ra y=1\in F$, hence $y\in F$, so that condition $(F_4)$ is satisfied. \\
$(b)\Rightarrow (a)$ Since $F$ is nonempty, $1\in F$. 
Let $x\in F$, $y\in X$ such that $x\ra y\in F$. 
By $(F_1)$ we get $x\odot (x\ra y)\in F$, and by Proposition \ref{qbe-120}$(4)$ we have $x\odot (x\ra y)\le y$. 
Applying $(F_4)$ we get $y\in F$, hence $F\in \mathcal{DS}(X)$. \\ 
$(b)\Rightarrow (c)$ Since $x\ra (x\Cup y)=x\ra ((x\ra y)\ra y)=(x\ra y)\ra (x\ra y)=1$, we have 
$x\le x\Cup y$. By $(F_4)$ we get $x\Cup y\in F$, hence $F$ satisfies condition $(F_3)$. \\
$(c)\Rightarrow (b)$ Suppose that $x\in F$ implies $x\Cup y\in F$, and let $y\in X$ such that $x\le y$, 
that is $x\ra y=1$. Since $x\Cup y=(x\ra y)\ra y=y$ and $x\Cup y\in F$, we get $y\in F$, that is $(F_4)$. \\
$(a)\Rightarrow (d)$ By Proposition \ref{fqbe-70}, $F\in \mathcal{DS}_q(X)$. 
Let $x\in F$ and $y\in X$ such that $x\sim y$. 
It follows that there exists $\alpha\in X$ such $x\le \alpha\le x$ and $y\le \alpha\le y$. 
From $x\in F$ and $x\ra \alpha=1\in F$, by $(DS_2)$ we have $\alpha\in F$, while $\alpha\in F$ and 
$\alpha\ra y=1$ imply $y\in F$. Thus condition $(F_5)$ is verified. \\
$(d)\Rightarrow (c)$ Let $F\in \mathcal{DS}_q(X)$, that is $F$ satisfies $(F_1)$ and $(F_2^{'})$. 
Let $x\in F$ and $y\in X$. Since $x\le_Q y\Cup x$, using $(F_2^{'})$ we get $y\Cup x\in F$. 
By Lemma \ref{fqbe-20}, $(y\Cup x)\sim (x\Cup y)$, and applying $(F_5)$ we get $x\Cup y\in F$, so that $(F_3)$ 
is verified. 
\end{proof}

\begin{corollary} \label{fqbe-80-10} If $X$ is commutative, then $\mathcal{DS}(X)=\mathcal{DS}_q(X)$. 
\end{corollary}
\begin{proof}
Since $X$ is commutative, $\le_Q = \le$, so that conditions $(F_2^{'})$ and $(F_4)$ coincide. 
It follows that $\mathcal{DS}(X)=\mathcal{DS}_q(X)$. 
\end{proof}

\begin{proposition} \label{fqbe-80-20} $\mathcal{DS}(X)=\mathcal{DS}_p(X)$. 
\end{proposition}
\begin{proof}
By Proposition \ref{fqbe-70}, $\mathcal{DS}(X)\subseteq \mathcal{DS}_q(X)\subseteq \mathcal{DS}_p(X)$. 
Conversely, let $F\in \mathcal{DS}_p(X)$, that is $F$ is a nonempty subset of $X$ satisfying conditions 
$(F_1)$-$(F_3)$. According to Proposition \ref{fqbe-80}, by $(F_1)$ and $(F_3)$, $F\in \mathcal{DS}(X)$. 
Hence $\mathcal{DS}_p(X)\subseteq \mathcal{DS}(X)$, that is $\mathcal{DS}(X)=\mathcal{DS}_p(X)$. 
\end{proof}

\begin{remark} \label{fqbe-90} If $F\in \mathcal{DS}(X)$, then $F$ is a subalgebra of $X$. \\
Indeed, ff $F\in \mathcal{DS}(X)$, $x\le_Q y\ra x$, $y\le_Q x\ra y$ imply $x\le y\ra x$, $y\le x\ra y$, that is 
$x\ra (y\ra x)=y\ra (x\ra y)=1\in F$. For $x,y\in F$, we get $x\ra y, y\ra x\in F$. 
Thus $F$ is a subalgebra of $X$. 
\end{remark}

The QW algebra from the next example is derived from an orthomodular lattice with six elements (see \cite{DvPu}). 

\begin{example} \label{fqbe-100} 
Let $X=\{0,a,b,c,d,1\}$ and let $(X,\ra,0,1)$ be the involutive BE algebra with $\ra$ and the corresponding 
operation $\Cap$ given in the following tables:  
\[
\begin{array}{c|ccccccc}
\ra & 0 & a & b & c & d & 1 \\ \hline
0   & 1 & 1 & 1 & 1 & 1 & 1 \\ 
a   & c & 1 & 1 & c & 1 & 1 \\ 
b   & d & 1 & 1 & 1 & d & 1 \\ 
c   & a & a & 1 & 1 & 1 & 1 \\
d   & b & 1 & b & 1 & 1 & 1 \\
1   & 0 & a & b & c & d & 1
\end{array}
\hspace{10mm}
\begin{array}{c|ccccccc}
\Cap & 0 & a & b & c & d & 1 \\ \hline
0    & 0 & 0 & 0 & 0 & 0 & 0 \\ 
a    & 0 & a & b & 0 & d & a \\ 
b    & 0 & a & b & c & 0 & b \\ 
c    & 0 & 0 & b & c & d & c \\
d    & 0 & a & 0 & c & d & d \\
1    & 0 & a & b & c & d & 1
\end{array}
.
\]
Then $X$ is a quantum-Wajsberg algebra and $\mathcal{DS}_q(X)=\{\{1\},\{a,1\},\{b,1\},\{c,1\},\{d,1\},X\}$, 
$\mathcal{DS}(X)=\{\{1\},X\}$. 
\end{example} 

Even though we work with deductive systems, sometimes we will use the notation $x\odot y=(x\ra y^*)^*$, just for 
easier computations. \\

Let $X$ be a QW algebra. For every subset $Y\subseteq X$, the smallest q-deductive system of $X$ containing $Y$ 
(i.e. the intersection of all q-deductive systems $F$ of $X$ such that $Y\subseteq F$) is called the 
\emph{q-deductive system generated} by $Y$ and it is denoted by $[Y)$. If $Y=\{x\}$ we write $[x)$ 
instead of $[\{x\})$ and $[x)$ is called a \emph{principal q-deductive system} of $X$. 
We can easily show that: \\
$\hspace*{1cm}$ $[Y)=\{y\in X\mid y\ge_Q y_1\odot y_2\odot \cdots \odot y_n$, for some $n\ge 1$ and 
$y_1,y_2,\dots,y_n\in Y\}$ and \\
$\hspace*{1cm}$ $[x)=\{y\in X\mid y\ge_Q x^n$, for some $n\ge 1\}$, for any $x\in X$. \\
For $F\in \mathcal{DS}_q(X)$ and $x\in X$, we have: \\
$\hspace*{1cm}$ $F_x=[F\cup \{x\})=\{y\in X\mid y\ge_Q f\odot x^n$, for some $n\ge 1$ and $f\in F\}$. \\
Obviously, if $x\in F$, then $F_x=F$. \\
We extend the notions of maximal and strongly maximal ideals in right-QMV algebras (\cite{Giunt7, DvPu}) to the 
case of QW algebras. 

\begin{definition} \label{fqbe-110} 
\emph{
A q-deductive system $F\in \mathcal{DS}_q(X)$ is said to be: \\
$(1)$ \emph{maximal} if it is proper and it is not contained in any other proper q-deductive system of $X;$ \\
$(2)$ \emph{strongly maximal} if, for all $x\in X$, $x\notin F$, there exists $n\ge 1$ such that $(x^n)^*\in F$. 
}
\end{definition}

\begin{proposition} \label{fqbe-120} If $F\in \mathcal{DS}_q(X)$, the following are equivalent: \\
$(a)$ $F$ is maximal; \\
$(b)$ for any $x\notin F$, there exist $f\in F$ and $n\ge 1$ such that $f\odot x^n=0$. 
\end{proposition}
\begin{proof}
$(a) \Rightarrow (b)$ Since $F$ is maximal and $x\notin F$, then $F_x=X$, hence $0\in F_x$. 
It follows that there exist $f\in F$ and $n\ge 1$ such that $f\odot x^n\le_Q 0$, that is $f\odot x^n= 0$. \\
$(b) \Rightarrow (a)$ Let $F^{'}$ be a proper q-deductive system of $X$ such that $F\subsetneq  F^{'}$. 
Then there exists $x\in F^{'}$ such that $x\notin F$. 
It follows that there exist $f\in F$ and $n\ge 1$ such that $f\odot x^n=0$. Since $f,x\in F^{'}$ we get 
$0\in F^{'}$, that is $F^{'}=X$. Hence $F$ is a maximal q-deductive system of $X$. 
\end{proof}

\begin{proposition} \label{fqbe-130} Every strongly maximal q-deductive system of $X$ is maximal.
\end{proposition}
\begin{proof}
Let $F$ be a strongly maximal q-deductive system of $X$ and let $F^{'}\in \mathcal{DS}_q(X)$ such that $F\subseteq F^{'}$. 
Suppose that there exists $x\in F^{'}$, $x\notin F$. It follows that there exists $n\ge 1$ such that $(x^n)^*\in F$. 
Since $x\in F^{'}$, we have $x^n\in F^{'}$, hence $0=x^n\odot (x^n)^*\in F^{'}$. 
Thus $F^{'}=X$, that is $F$ is maximal.
\end{proof}

\begin{proposition} \label{fqbe-140} If $\mathcal{DS}_q(X)=\mathcal{DS}(X)$, then every maximal q-deductive system 
of $X$ is strongly maximal.  
\end{proposition}
\begin{proof}
Let $X$ be a maximal q-deductive system of $X$ and let $x\in F$ such that $x\notin F$. 
Then $F_x=X$, hence $0\in F_x$. 
It follows that there exist $n\ge 1$ and $f\in F$ such that $f\odot x^n\le_Q 0$, so that $f\odot x^n\le 0$. 
By Proposition \ref{qbe-120}, we have $f\le x^n\ra 0=(x^n)^*$. 
Since $F\in \mathcal{DS}(X)$, by $(F_4)$ we get $(x^n)^*\in F$, hence $F$ is strongly maximal. 
\end{proof}

\begin{corollary} \label{fqbe-150} If $X$ is commutative, then every maximal q-deductive system of $X$ is strongly maximal.  
\end{corollary}

\begin{remark} \label{fqbe-160} 
$(1)$ In general, a maximal q-deductive system of $X$ need not be strongly maximal. 
Indeed, consider the QW algebra $X$ from Example \ref{fqbe-100} and the maximal q-deductive system $F=\{a,1\}$. 
We have $b^n=b$ for all $n\ge 1$, so that $(b^n)^*=b^*=d\notin F$. Hence $F$ is not strongly maximal. \\
$(2)$ The same remark for a maximal deductive system of $X$: the maximal deductive system $F=\{1\}$ from 
Example \ref{fqbe-100} is not strongly maximal. 
\end{remark}

\begin{remark} \label{fqbe-170} A q-deductive system $F$ of a BE algebra $(X,\ra,1)$ is called \emph{commutative} if 
for all $x,y\in X$, $y\ra x\in F$ implies $((x\ra y)\ra y)\ra x\in F$. 
This notion plays an important role in the study of states, measures, internal states and valuations on $X$, 
since the kernels of these maps are commutative q-deductive system (see \cite{Ciu61}).    
In the case of a QW algebra $X$, due to Proposition \ref{qbe-60}$(8)$ we have $(x\Cup y)\ra x=y\ra x$, for all 
$x,y\in X$, so that every q-deductive system of $X$ is commutative. 
\end{remark}

$\vspace*{5mm}$

\section{Congruences and quotient quantum-Wajsberg algebras}

In this section, we introduce the notion of congruences determined by deductive systems of a quantum-Wajsberg algebra, 
and we show that there is a relationship between congruences and deductive systems. 
We also define the quotient quantum-Wajsberg algebra with respect to a deductive system $F$, and prove that the 
quotient quantum-Wajsberg algebra is locally finite if and only if $F$ is strongly maximal. 
In what follows, $(X,\ra,^*,1)$ will be a quantum-Wajsberg algebra, unless otherwise stated.  

\begin{definition} \label{cqbe-10} Let $F\in \mathcal{DS}_q(X)$. For $x,y\in X$, $x\equiv_F y$ if and only if 
there exists $\lambda\in X$ such that $x,y\le_Q \lambda$ and $\lambda\ra x, \lambda\ra y\in F$. 
\end{definition}

\begin{proposition} \label{cqbe-20} If $F\in \mathcal{DS}(X)$, the following are equivalent for all $x,y\in X$: \\
$(a)$ $x\equiv_F y;$ \\
$(b)$ there exist $\alpha, \beta \in F$ such that $x\le_Q \alpha$, $y\le_Q \beta$ and $\alpha\ra x=\beta\ra y;$ \\
$(c)$ there exist $\alpha, \beta \in F$ such that $x\le_Q \beta\ra y$ and $y\le_Q \alpha\ra x$. 
\end{proposition}
\begin{proof}
$(a)\Rightarrow (b)$ Suppose that $x\equiv_F y$, so that there exists $\lambda\in X$ such that $x,y\le_Q \lambda$ 
and $\lambda\ra x, \lambda\ra y\in F$. Take $\alpha:=\lambda\ra x$, $\beta:=\lambda\ra y$. 
By Proposition \ref{qbe-60}$(2)$, $x\le_Q \alpha$ and $y\le_Q \beta$. 
Using Proposition \ref{qbe-70}$(1)$, $\lambda=\lambda\Cup x=\lambda\Cup y$. 
It follows that $\alpha\ra x=(\lambda\ra x)\ra x=\lambda\Cup x=\lambda=\lambda\Cup y=(\lambda\ra y)\ra y=\beta\ra y$. \\
$(b)\Rightarrow (a)$ Assume that there exist $\alpha, \beta \in F$ such that $x\le_Q \alpha$, $y\le_Q \beta$ and 
$\alpha\ra x=\beta\ra y$. 
Taking $\lambda:=\alpha\ra x=\beta\ra y$, we have $x\le_Q \lambda$ and $y\le_Q \lambda$. 
Moreover $\lambda\ra x=(\alpha\ra x)\ra x=\alpha\Cup x=\alpha\in F$ (since $x\le_Q \alpha)$ and 
$\lambda\ra y=(\beta\ra y)\ra y=\beta\Cup y=\beta\in F$ (since $y\le_Q \beta$). 
Hence $x\equiv_F y$. \\
$(b)\Rightarrow (c)$ With $\alpha, \beta$ from $(b)$ we have $x\le_Q \alpha\ra x=\beta\ra y$ and 
$y\le_Q \beta\ra y=\alpha\ra x$, hence condition $(c)$ is satisfied. \\
$(c)\Rightarrow (b)$ With $\alpha, \beta$ from $(c)$, take $\gamma:=\alpha\Cup x$ and $\delta:=(\alpha\ra x)\ra y$. 
Obviously $x\le_Q \gamma$ and $y\le_Q \delta$.  
Since $F\in \mathcal{DS}(X)$ and $\alpha\in F$, by Proposition \ref{fqbe-80} we get $\gamma=\alpha\Cup x\in F$. 
From $\beta\ra y\ge_Q x$ we have 
$\beta\ra \delta=\beta\ra ((\alpha\ra x)\ra y)=(\alpha\ra x)\ra (\beta\ra y)\ge_Q (\alpha\ra x)\ra x=
\alpha\Cup x=\gamma\in F$. 
It follows that $\beta\ra \delta\in F$.  
From $F\in \mathcal{DS}(X)$ and $\beta, \beta\ra \delta\in F$, we get $\delta \in F$. 
Moreover $\gamma\ra x=(\alpha\Cup x)\ra x=\alpha\ra x$ (by Proposition \ref{qbe-60}$(9)$) and 
$\delta\ra y=((\alpha\ra x)\ra y)\ra y=(\alpha\ra x)\Cup y=\alpha \ra x$, since $y\le_Q \alpha\ra x$.  
Hence $\gamma\ra x=\delta\ra y=\alpha\ra x$. 
Redefining $\alpha:=\gamma$ and $\beta:=\delta$, the proof of $(b)$ is complete. 
\end{proof}

\begin{lemma} \label{cqbe-20-10} Let $x,y\in X$. \\
$(1)$ if $F\in \mathcal{DS}_q(X)$ and $x\equiv_F y$, then $x\ra y, y\ra x\in F;$ \\
$(2)$ if $F\in \mathcal{DS}(X)$, $x\equiv_F y$ and $x\in F$, then $y\in F$.  
\end{lemma}
\begin{proof}
$(1)$ Since $x\equiv_F y$, there exists $\alpha\in X$ such that $x,y\le_Q \alpha$ and $\alpha\ra x, \alpha\ra y\in F$. 
It follows that $\alpha\ra y\le_Q x\ra y$ and $\alpha\ra x\le_Q y\ra x$, hence $x\ra y, y\ra x\in F$. \\ 
$(2)$ Since $\mathcal{DS}(X)\subseteq \mathcal{DS}_q(X)$, using $(1)$, $x\equiv_F y$ implies $x\ra y\in F$. 
From $x, x\ra y\in F$ we get $y\in F$. 
\end{proof}

\begin{proposition} \label{cqbe-20-20} Let $X$ be a commutative QW algebra, let $F\in \mathcal{DS}(X)$ and 
let $x,y\in X$. Then $x\equiv_F y$ if and only if $x\ra y, y\ra x\in F$. 
\end{proposition}
\begin{proof}
If $x\equiv_F y$, by Lemma \ref{cqbe-20-10}, $x\ra y, y\ra x\in F$.
Conversely, let $x,y\in X$ such that $x\ra y, y\ra x\in F$. 
Since $X$ is commutative, then it is a commutative BCK algebra and $\le_Q$ coincides with $\le$. 
Taking $\alpha:=y\ra x$, $\beta:=x\ra y$, we have $\alpha, \beta \in F$ and 
$x\le (x\ra y)\ra y=\beta\ra y$, $y\le (y\ra x)\ra x=\alpha\ra x$. 
Hence $x\le_Q \beta\ra y$, $y\le_Q \alpha\ra x$, and by Proposition \ref{cqbe-20}$(c)$, $x\equiv_F y$. 
\end{proof}

\begin{proposition} \label{cqbe-30} Let $F\in \mathcal{DS}(X)$. The relation $\equiv_F$ is an equivalence relation 
on $X$. 
\end{proposition}
\begin{proof}
Obviously $\equiv_F$ is reflexive and symmetric. \\
Suppose that $x\equiv_F y$ and $y\equiv_F z$. 
It follows that there exist $\alpha, \beta \in X$ such that $x,y\le_Q \alpha$, $\alpha\ra x, \alpha\ra y\in F$ 
and $y,z\le_Q \beta$, $\beta\ra y, \beta\ra z\in F$. 
Take $\gamma:=(\alpha\ra x)\odot (\beta\ra y)\in F$ and $\delta:=(\alpha\ra y)\odot (\beta\ra z)\in F$. 
We show that $x\le_Q \delta\ra z$ and $z\le_Q \gamma\ra x$. Indeed, we have: \\
$\hspace*{2.00cm}$ $\delta\ra z=z^*\ra \delta^*=z^*\ra ((\alpha\ra y)\odot (\beta\ra z))^*$ \\
$\hspace*{3.10cm}$ $=z^*\ra ((\alpha\ra y)\ra (\beta\ra z)^*)=(\alpha\ra y)\ra (z^*\ra (\beta\ra z)^*)$ \\
$\hspace*{3.10cm}$ $=(\alpha\ra y)\ra ((\beta\ra z)\ra z)=(\alpha\ra y)\ra (\beta\Cup z)$ \\
$\hspace*{3.10cm}$ $=(\alpha\ra y)\ra \beta\ge_Q (\alpha\ra y)\ra y=\alpha\Cup y=\alpha\ge_Q x$. \\
(Since $z\le_Q \beta$ implies $\beta\Cup z=\beta$, $\beta\ge_Q y$ implies 
$(\alpha\ra y)\ra \beta\ge_Q (\alpha\ra y)\ra y$, and $y\le_Q \alpha$ implies $\alpha\Cup y=\alpha$). 
Similarly: \\
$\hspace*{2.00cm}$ $\gamma\ra x=x^*\ra \gamma^*=x^*\ra ((\alpha\ra x)\odot (\beta\ra y))^*$ \\
$\hspace*{3.10cm}$ $=x^*\ra ((\alpha\ra x)\ra (\beta\ra y)^*)$ \\
$\hspace*{3.10cm}$ $=x^*\ra ((\beta\ra y)\ra (\alpha\ra x)^*)=(\beta\ra y)\ra (x^*\ra (\alpha\ra x)^*)$ \\
$\hspace*{3.10cm}$ $=(\beta\ra y)\ra ((\alpha\ra x)\ra x)=(\beta\ra y)\ra (\alpha\Cup x)$ \\
$\hspace*{3.10cm}$ $=(\beta\ra y)\ra \alpha\ge_Q (\beta\ra y)\ra y=\beta\Cup y=\beta\ge_Q z$. \\
(Since $x\le_Q \alpha$ implies $\alpha\Cup x=\alpha$, $\alpha\ge_Q y$ implies 
$(\beta\ra y)\ra \alpha\ge_Q (\beta\ra y)\ra y$, and $y\le_Q \beta$ implies $\beta\Cup y=\beta)$. 
Applying Proposition \ref{cqbe-20}$(c)$, it follows that $x\equiv_F z$, that is $\equiv_F$ is transitive, so that 
it is an equivalence relation on $X$. 
\end{proof}

\begin{proposition} \label{cqbe-40} Let $F\in \mathcal{DS}(X)$ and let $x,y\in X$. 
If $x\equiv_F y$, then $x^*\equiv_F y^*$.  
\end{proposition}
\begin{proof}
Since $x\equiv_F y$, there exists $\alpha\in X$ such that $x,y\le_Q \alpha$ and $\alpha\ra x, \alpha\ra y\in F$. 
It follows that $x^*\ge_Q \alpha^*$ and $y^*\ge_Q \alpha^*$. 
Take $\gamma:=x\ra (\alpha\odot y^*)$, so that, by Proposition \ref{qbe-60}$(2)$, $\gamma\ge_Q x^*$. 
From $x\le_Q \alpha$, by Proposition \ref{qbe-70}$(3)$ we get: \\
$\hspace*{2.00cm}$ $\gamma=x\ra (\alpha\odot y^*)\ge_Q \alpha\ra (\alpha\odot y^*)
                          =(\alpha\Cup y)\ra (\alpha\odot y^*)$ \\
$\hspace*{2.30cm}$ $=((\alpha\ra y)\ra y)\ra (\alpha\odot y^*)=(y^*\ra (\alpha\ra y)^*)\ra (\alpha\odot y^*)$ \\
$\hspace*{2.30cm}$ $=(y^*\ra (\alpha\odot y^*))\ra (\alpha\odot y^*)=y^*\Cup (\alpha\odot y^*)=y^*$. \\
(Since $y\le_Q \alpha$ implies $\alpha=\alpha\Cup y$, and $\alpha\odot y^*\le_Q y^*$ 
implies $y^*=y^*\Cup (\alpha\odot y^*)$). Hence $\gamma\ge_Q y^*$. 
We also have: \\
$\hspace*{2.00cm}$ $\gamma\ra x^*=(x\ra (\alpha\odot y^*))\ra x^*=((\alpha\odot y^*)^*\ra x^*)\ra x^*$ \\
$\hspace*{3.30cm}$ $=(\alpha\odot y^*)^*\Cup x^*=(\alpha\ra y)\Cup x^*\in F$, \\
since $\alpha\ra y\in F$, and $F\in \mathcal{DS}(X)$. \\
By Proposition \ref{qbe-20}$(2)$,$(3)$, $x\le_Q \alpha$ implies $x=(\alpha\ra x)\odot \alpha$ and 
$x\ra (\alpha\odot y^*)=((\alpha\ra y)\odot x)^*$. 
From $\alpha\ge_Q y$ we get: \\  
$\hspace*{2.00cm}$ $(\alpha\ra x)\odot (\alpha\ra y)\odot \alpha \ge_Q (\alpha\ra x)\odot (\alpha\ra y)\odot y$ and \\ 
$\hspace*{2.00cm}$ $((\alpha\ra x)\odot (\alpha\ra y)\odot \alpha)^*\ra y^* 
                   \ge_Q ((\alpha\ra x)\odot (\alpha\ra y)\odot y)^*\ra y^*$. \\
It follows that: \\
$\hspace*{2.00cm}$ $\gamma\ra y^*=(x\ra (\alpha\odot y^*))\ra y^*=((\alpha\ra y)\odot x)^*\ra y^*$ \\
$\hspace*{3.30cm}$ $=((\alpha\ra y)\odot (\alpha\ra x)\odot \alpha)^*\ra y^*$ \\
$\hspace*{3.30cm}$ $\ge_Q ((\alpha\ra x)\odot (\alpha\ra y)\odot y)^*\ra y^*$ \\
$\hspace*{3.30cm}$ $=(((\alpha\ra x)\odot (\alpha\ra y))\ra y^*)\ra y^*$ \\
$\hspace*{3.30cm}$ $=((\alpha\ra x)\odot (\alpha\ra y))\Cup y^*\in F$. \\
since $(\alpha\ra x)\odot (\alpha\ra y)\in F$, and $F\in \mathcal{DS}(X)$.  
We proved that there exists $\gamma\in X$ such that $x^*, y^*\le_Q \gamma$, and 
$\gamma\ra x^*, \gamma\ra y^*\in F$, hence $x^*\equiv_F y^*$. 
\end{proof}

\begin{proposition} \label{cqbe-50} Let $F\in \mathcal{DS}(X)$ and let $x,y\in X$. 
If $x\equiv_F y$ and $u\equiv_F v$, then $x\odot u\equiv_F y\odot v$.  
\end{proposition}
\begin{proof}
By hypothesis there exist $\alpha,\beta\in X$ such that $x,y\le_Q \alpha$, $u,v\le_Q \beta$ and 
$\alpha\ra x, \alpha\ra y, \beta\ra u, \beta\ra v\in F$. 
Take $\gamma:=\alpha\odot \beta$, and we can easily check that $x\odot u, y\odot v\le_Q \gamma$. 
Using the commutativity and associativity of $\odot$, as well as the Proposition \ref{qbe-20-10}, we get: \\ 
$\hspace*{1.50cm}$ $((\alpha\ra x)\odot (\beta\ra u))\ra (\gamma\ra (x\odot u))=$ \\
$\hspace*{4.00cm}$ $=((\alpha\ra x)\odot (\beta\ra u))\ra ((\alpha\odot \beta)\ra (x\odot u))$ \\
$\hspace*{4.00cm}$ $=(\alpha\ra x)\odot (\beta\ra u)\ra ((\alpha\odot \beta)\odot (x\odot u)^*)^*$ \\           $\hspace*{4.00cm}$ $=[(\alpha\ra x)\odot (\beta\ra u)\odot (\alpha\odot \beta)\odot (x\odot u)^*]^*$ \\         $\hspace*{4.00cm}$ $=[((\alpha\ra x)\odot \alpha)\odot ((\beta\ra u)\odot \beta)\odot (x\odot u)^*]^*$ \\           $\hspace*{4.00cm}$ $=[(x\odot u)\odot (x\odot u)^*]^*
                    =(x\odot u)\ra (x\odot u)=1$. \\                   
From $(\alpha\ra x)\odot (\beta\ra u)\ra (\gamma\ra (x\odot u))=1\in F$,  
$(\alpha\ra x)\odot (\beta\ra u)\in F$ and $F\in \mathcal{DS}(X)$, we get $\gamma\ra (x\odot u)\in F$. 
Similarly $\gamma\ra (y\odot v)\in F$, hence $x\odot u\equiv_F y\odot v$.  
\end{proof}

\begin{theorem} \label{cqbe-60} For any $F\in \mathcal{DS}(X)$, there exists a congruence relation $\equiv_F$ 
on X such that $\{x\in X\mid x\equiv_F 1\}=F$. Conversely, given a congruence relation $\equiv$ on $X$, then 
$\{x\in X\mid x\equiv 1\}\in \mathcal{DS}(X)$. 
\end{theorem}
\begin{proof}
Let $F\in \mathcal{DS}(X)$. According to Proposition \ref{cqbe-30}, $\equiv_F$ is an equivalence relation on $X$. 
We show that $x\equiv_F y$ and $u\equiv_F v$ imply $x\ra u\equiv_F y\ra v$.  
By Proposition \ref{cqbe-40}, $u^*\equiv_F v^*$, and applying Proposition \ref{cqbe-50}, $x\equiv_F y$ and 
$u^*\equiv_F v^*$ imply $x\odot u^*\equiv_F y\odot v^*$. 
Using again Proposition \ref{cqbe-40}, we get $(x\odot u^*)^*\equiv_F (y\odot v^*)^*$, 
that is $x\ra u\equiv_F y\ra v$. 
Hence $\equiv_F$ is a congruence on $X$. \\
Let $x\in \{x\in X\mid x\equiv_F 1\}$, so that there exist $\alpha,\beta\in F$ such that $x\le_Q \alpha$, 
$1\le_Q \beta$ and $\alpha\ra x=\beta \ra 1=1$. 
Since $x\le_Q \alpha$, we have $\alpha\Cup x=\alpha$, hence 
$x=1\ra x=(\alpha\ra x)\ra x=\alpha\Cup x=\alpha\in F$. 
Thus $\{x\in X\mid x\equiv_F 1\}\subseteq F$. 
If $x\in F$, for $\alpha:=x$, $\beta:=1$ we have $\alpha,\beta \in F$, $x\le_Q \alpha$, $1\le_Q \beta$, 
and $\alpha\ra x=\beta\ra 1=1$. 
Hence $x\equiv_F 1$, so that $F\subseteq \{x\in X\mid x\equiv_F 1\}$. 
It follows that $\{x\in X\mid x\equiv_F 1\}= F$. \\
Conversely, let $\equiv$ be a congruence on $X$, and let $F=\{x\in X\mid x\equiv 1\}$. 
If $x,y \in F$, then $x\equiv 1$ and $y\equiv 1$. 
From $y\equiv 1$ and $0\equiv 0$ we have $y\ra 0\equiv 1\ra 0$, that is $y^*\equiv 0$. 
Moreover $x\equiv 1$ and $y^*\equiv 0$ imply $x\ra y^*\equiv 1\ra 0=0$, while 
$x\ra y^*\equiv 0$ and $0\equiv 0$ imply $(x\ra y^*)\ra 0\equiv 0\ra 0=1$. 
Hence $x\odot y=(x\ra y^*)^*\equiv 1$, that is $x\odot y\in F$ and condition $(F_1)$ is satisfied. 
Let $x\in F$ and $y\in X$, so that $x\equiv 1$. 
From $x\equiv 1$ and $y\equiv y$ we get $x\ra y\equiv 1\ra y=y$. 
Moreover $x\ra y\equiv y$ and $y\equiv y$ imply $(x\ra y)\ra y\equiv y\ra y=1$. 
It follows that $x\Cup y\equiv 1$, hence $x\Cup y\in F$, that is condition $(F_3)$ is also verified. 
Taking into consideration Proposition \ref{fqbe-80}, we conclude that $F \in \mathcal{DS}(X)$.
\end{proof}

\begin{corollary} \label{cqbe-70} For any $F\in \mathcal{DS}(X)$, $x\equiv_F y$ and $u\equiv_F v$ imply 
$x\Cup u\equiv_F y\Cup v$ and $x\Cap u\equiv_F y\Cap v$. 
\end{corollary} 

\noindent
The quotient QW algebra induced by $\equiv_F$ is denoted by $X/F$. \\
\noindent
For any $x\in X$, the smallest $n\in \mathbb{N}$ such that $x^n=0$ is called the \emph{order} of $x$, and 
it is denoted by $\ord(x)$. If there is no such $n$, then $\ord(x)=\infty$. 
A QW algebra $X$ is called \emph{locally finite} if any $x\in X$, $x\ne 1$, has a finite order. 

\begin{theorem} \label{cqbe-80} Let $F\in \mathcal{DS}(X)$. Then the following are equivalent: \\
$(a)$ $F$ is strongly maximal; \\
$(b)$ $X/F$ is locally finite.  
\end{theorem}
\begin{proof}
$(a) \Rightarrow (b)$ Let $F$ be a strongly maximal deductive system of $X$, and let $x\not\equiv_F 1$, 
so that $x\notin F$. 
It follows that there exists $n\ge 1$ such that $(x^n)^*\in F$, hence $(x^n)^* \equiv_F 1$. 
Thus $x^n\equiv_F 0$, that is $X/F$ is locally finite. \\
$(b) \Rightarrow (a)$ Suppose that $X/F$ is locally finite, and let $x\notin F$, that is $x\not\equiv _F 1$. 
Hence there is $n\ge 1$ such that $x^n\equiv_F 0$, so that $(x^n)^*\equiv_F 1$. 
It follows that $(x^n)^*\in F$, thus $F$ is strongly maximal. 
\end{proof}

$\vspace*{5mm}$

\section{On the linearity of quantum-Wajsberg algebras}

We investigate the linearity of quantum-Wajsberg algebras by redefining the notion of weakly linearity introduced 
by R. Giuntini in \cite{Giunt6} for quantum-MV algebras. 
We also prove certain properties of weakly linear QW algebras, and define the notion of prime deductive systems. 
Furthermore, we prove that the quotient of a quantum-Wajsberg algebra with respect to a deductive system is weakly linear if and only if the deductive system is prime. 
Finally, we define the quasi-linear quantum-Wajsberg algebras, and give equivalent conditions for this notion.  
In what follows, $X$ will be a quantum-Wajsberg algebra, unless otherwise stated. \\
We can see that, in the case of proper QW algebras, the poset $(X,\le_Q)$ is not linearly ordered. 
Indeed, suppose that, for all $x,y\in X$, $x\le_Q y$ or $y\le_Q x$. 
Then by Proposition \ref{qbe-20}$(1)$, we get $x\Cap y=y\Cap x=x$ or $y\Cap x=x\Cap y=y$, respectively. 
It is a contradiction, since, in general, the operations $\Cup$ and $\Cap$ are not commutative ($(X,\Cap,\Cup)$ 
is not a lattice). 

\begin{definition} \label{wlqbe-10}
\emph{
A quantum-Wajsberg algebra $X$ is \emph{weakly linear} if it is linearly ordered with respect to $\le$, that is, 
for all $x,y\in X$, $x\ra y=1$ or $y\ra x=1$.  
}
\end{definition}

\begin{proposition} \label{wlqbe-20} Given a QW algebra $X$, the following are equivalent: \\ 
$(a)$ $X$ is weakly linear; \\
$(b)$ for all $x,y\in X$, $x\Cup y=y$ or $y\Cup x=x;$ \\
$(c)$ for all $x,y\in X$, $x\Cap y=y$ or $y\Cap x=x$. 
\end{proposition}
\begin{proof}
$(a)\Rightarrow (b)$ Assume that $X$ is weakly linear and let $x,y\in X$. 
If $x\ra y=1$, then $x\Cup y=(x\ra y)\ra y=y$, while $y\ra x=1$ implies $y\Cup x=(y\ra x)\ra x=x$. \\
$(b)\Rightarrow (a)$ Let $x,y \in X$. 
If $x\Cup y=y$, since by Proposition \ref{qbe-60}$(9)$, $(x\Cup y)\ra y=x\ra y$, we get $x\ra y=y\ra y=1$. 
Similarly, if $y\Cup x=x$, from $(y\Cup x)\ra x=y\ra x$ we have $y\ra x=1$. \\
$(a)\Rightarrow (c)$ Let $X$ be weakly linear and let $x,y\in X$. 
If $x\ra y=1$, we have $y\Cap x=((y^*\ra x^*)\ra x^*)^*=((x\ra y)\ra x^*)^*=x$, and similarly 
$y\ra x=1$ implies $x\Cap y=((y\ra x)\ra y^*)^*=y$. \\
$(c)\Rightarrow (a)$ Consider $x,y\in X$. 
If $x\Cap y=y$, from $y\ra (x\Cap y)=y\ra x$ (by Proposition \ref{qbe-60}$(1)$) we get $y\ra x=1$. 
Similarly, from $y\Cap x=x$, using $x\ra (y\Cap x)=x\ra y$ we have $x\ra y=1$. 
\end{proof}

\begin{proposition} \label{wlqbe-30} Let $X$ be a weakly linear  QW algebra. 
The following hold for all $x,y,z\in X$: \\ 
$(1)$ $((x\ra y)\ra z)\Cap ((y\ra x)\ra z)=z;$ \\
$(2)$ $x\Cap y=0$ implies $(z\ra x)\Cap (z\ra y)=z^*;$ \\
$(3)$ $x\Cap (y\Cap z)=0$ implies $x\Cap (z\Cap y)=0;$ \\
$(4)$ $x\sim y$ and $x\Cup y=1$ imply $x=y=1$.    
\end{proposition}
\begin{proof}
$(1)$ Since $X$ is weakly linear, for all $x,y\in X$ we have $x\ra y=1$ or $y\ra x=1$. 
For $x\ra y=1$, we get that $((x\ra y)\ra z)\Cap ((y\ra x)\ra z)=z\Cap ((y\ra x)\ra z)=z$, 
since by Proposition \ref{qbe-60}$(2)$, $z\le_Q (y\ra x)\ra z$. 
Similarly $x\ra y=1$ implies $((x\ra y)\ra z)\Cap ((x\ra y)\ra z)\Cap z)=((x\ra y)\ra z)\Cap z=z$, 
by Proposition \ref{qbe-20}$(1)$, since $z\le_Q (x\ra y)\ra z$. \\
$(2)$ By Proposition \ref{qbe-60}$(1)$ we have $x\ra (y\Cap x)=x\ra y$ and $y\ra (x\Cap y)=y\ra x$, and 
by Proposition \ref{qbe-100}$(3)$, $x\Cap y=0$, implies $y\Cap x=0$. 
It follows that $x\ra y=x^*$ and $y\ra x=y^*$, so that $(x\ra y)^*=x$ and $(y\ra x)^*=y$. 
Using $(1)$, we get: 
$(z\ra x)\Cap (z\ra y)=(z\ra (x\ra y)^*)\Cap (z\ra (y\ra x)^*)=((x\ra y)\ra z^*)\Cap ((y\ra x)\ra z^*)=z^*$. \\
$(3)$ Since $x^*\le x^*\Cup (z\Cap y)^*$, we have $x^*\ra (x^*\Cup (z\Cap y)^*)=1$, and so 
$(x^*\Cup (z\Cap y)^*)^*\ra x=x^*\ra (x^*\Cup (z\Cap y)^*)=1$. 
Using Proposition \ref{qbe-100}$(1)$, we get \\
$\hspace*{1.00cm}$ $(x^*\Cup (z\Cap y)^*)^*\ra (y\Cap z)=(y\Cap z)^*\ra (x^*\Cup (z\Cap y)^*)$ \\
$\hspace*{5.40cm}$ $=(y\Cap z)^*\ra ((x^*\ra (z\Cap y)^*)\ra (z\Cap y)^*)$ \\
$\hspace*{5.40cm}$ $=(x^*\ra (z\Cap y)^*)\ra ((y\Cap z)^*\ra (z\Cap y)^*)$ \\
$\hspace*{5.40cm}$ $=(x^*\ra (z\Cap y)^*)\ra 1=1$. \\
Since $x\Cap (y\Cap z)=0$, using $(2)$ it follows that \\
$\hspace*{2.00cm}$ $1=((x^*\Cup (z\Cap y)^*)^*\ra x)\Cap ((x^*\Cup (z\Cap y)^*)^*\ra (y\Cap z))=x^*\Cup (z\Cap y)^*$. \\
Hence $x^*\Cup (z\Cap y)^*=1$, so that  $(x\Cap (z\Cap y))^*=1$, that is $x\Cap (z\Cap y)=0$. \\
$(4)$ By Proposition \ref{qbe-100}$(4)$, $x\Cup y=1$ implies $y\Cup x=1$. 
Using Proposition \ref{qbe-60}$(1)$, $(x^*\ra y^*)\ra (y^*\Cap x^*)=x^*$ and $(y^*\ra x^*)\ra (x^*\Cap y^*)=y^*$. 
It follows that $(x\ra y)\ra (x\Cup y)^*=y^*$ and $(y\ra x)\ra (y\Cup x)^*=x^*$.
Since $x\Cup y=y\Cup x=1$, we get $x\ra y=y$ and $y\ra x=x$.  
From $x\sim y$, there exists $\alpha\in X$ such that $x\le_Q \alpha\le_Q x$ and $y\le_Q \alpha\le_Q y$, 
that is $x\ra \alpha =1$ and $y\ra \alpha =1$. 
By $(1)$, we have $((x\ra y)\ra \alpha)\Cap ((y\ra x)\ra \alpha)=\alpha$, hence 
$(y\ra \alpha)\Cap (x\ra \alpha)=\alpha$. 
It follows that $\alpha=1$, and so $x=y=1$. 
\end{proof}

The notion of a prime deductive system is used to characterize the weakly linear quotient of a QW algebra. 

\begin{definition} \label{wlqbe-40} 
\emph{
$F\in \mathcal {DS}(X)$ is called a \emph{prime deductive system} if, for all $x,y\in X$, $x\ra y\in F$ or 
$y\ra x\in F$. 
}
\end{definition}

\begin{proposition} \label{wlqbe-50} Let $F\in \mathcal {DS}(X)$. The following are equivalent: \\
$(a)$ $F$ is prime; \\
$(b)$ for all $x,y\in X$, $x\Cup y\in F$ implies $x\in F$ or $y\in F$. 
\end{proposition}
\begin{proof}
$(a)\Rightarrow (b)$ Let $F\in \mathcal {DS}(X)$, $F$ prime and let $x,y \in X$. Then $x\ra y\in F$ or $y\ra x\in F$, 
and suppose that $x\Cup y\in F$. 
Since by Proposition \ref{qbe-100}, $(x\Cup y)\ra (y\Cup x)=(y\Cup x)\ra (x\Cup y)=1\in F$, we have 
$x\Cup y\in F$ if and only if $y\Cup x\in F$. 
Let us consider two cases: \\   
$(i)$ if $x\ra y\in F$, $x\Cup y\in F$, then $y=(x\Cup y)\odot (x\ra y)\in F$  
(since $y\le_Q x\ra y$, by Proposition \ref{qbe-20-10}$(2)$ we have $((x\ra y)\ra y)\odot (x\ra y)=y)$. \\
$(ii)$ if $y\ra x\in F$, $x\Cup y\in F$, then $y\Cup x\in F$ and $x=(y\Cup x)\odot (y\ra x)\in F$. \\
$(b)\Rightarrow (a)$ Let $x,y\in X$ such that $x\Cup y\in F$ implies $x\in F$ or $y\in F$. 
Since by Proposition \ref{qbe-80}$(7)$, $(x\ra y)\Cup (y\ra x)=1\in F$, we get $x\ra y\in F$ or $y\ra x\in F$, 
that is $F$ is prime. 
\end{proof}

\begin{theorem} \label{wlqbe-60} Let $F\in \mathcal {DS}(X)$. The following are equivalent: \\
$(a)$ $F$ is prime; \\
$(b)$ $X/F$ is weakly linear.  
\end{theorem}
\begin{proof}
$(a)\Rightarrow (b)$ Let $F$ be a prime deductive system of $X$ and let $x,y\in X$. 
Suppose that $x\Cup y\not\equiv_F y$, that is $(x\Cup y)\ra y\not\equiv_F 1$, so that $x\ra y\not\equiv_F 1$ 
(by Proposition \ref{qbe-60}$(9)$, $(x\Cup y)\ra y=x\ra y$). 
It follows that $x\ra y\notin F$. Since $F$ is prime, we have $y\ra x\in F$. 
Hence $(y\Cup x)\ra x\in F$, that is $(y\Cup x)\ra x\equiv_F 1$, so that $((y\Cup x)\ra x)\ra x\equiv_F 1\ra x=x$. 
It follows that $(y\Cup x)\Cup x\equiv_F x$. Since by Proposition \ref{qbe-60}$(5)$, 
$(y\Cup x)\Cup x=y\Cup x$, we get $y\Cup x\equiv_F x$. Then by Proposition \ref{wlqbe-20}, $X/F$ is weakly linear. \\
$(b)\Rightarrow (a)$ Suppose that $X/F$ is weakly linear, that is, for all $x,y\in X$, $x\Cup y\equiv_F y$ or 
$y\Cup x\equiv_F x$. It follows that $(x\Cup y)\ra y\equiv_F 1$ or $(y\Cup x)\ra x\equiv_F 1$, hence 
$(x\Cup y)\ra y\in F$ or $(y\Cup x)\ra x\in F$. 
Thus $x\ra y\in F$ or $y\ra x\in F$, so that $F$ is a prime deductive system. 
\end{proof}

\begin{corollary} \label{wlqbe-70} If $X$ is commutative and $F\in \mathcal {DS}(X)$, then $F$ is prime if and 
only if $X/F$ is linearly ordered. 
\end{corollary}

\begin{definition} \label{qlqbe-10} 
\emph{
A QW algebra $X$ is said to be \emph{quasi-linear} if, for all $x,y\in X$, $x\nleq_Q y$ implies $y< x$. 
}
\end{definition}

\begin{proposition} \label{qlqbe-20} The following are equivalent: \\
$(a)$ $X$ is quasi-linear; \\
$(b)$ for all $x,y\in X$, $x\nleq y$ implies $y<_Q x;$ \\
$(c)$ for all $x,y,z\in X$, $z\ra x=z\ra y\ne 1$ implies $x=y$.  
\end{proposition}
\begin{proof}
$(a)\Rightarrow (b)$ Let $x,y\in X$ such that $x\nleq y$, and suppose $y\nless_Q x$. 
If $x=y$, then $x\ra y=1$, that is $x\le y$, hence we assume that $x\ne y$. 
Using $(a)$, $y\nless_Q x$ implies $x< y$, so that, by Proposition \ref{qbe-60}$(10)$, we have $x=y\Cap x$. 
It follows that: 
$x\ra y=(y\Cap x)\ra y=y^*\ra (y\Cap x)^*=y^*\ra (y^*\Cup x^*)=y^*\ra ((y^*\ra x^*)\ra x^*)=
(y^*\ra x^*)\ra (y^*\ra x^*)=1$, hence $x\le y$, a contradiction. 
Hence $y<_Q x$. \\
$(b)\Rightarrow (c)$ $z\ra x=z\ra y\ne 1$ implies $z\nleq x$ and $z\nleq y$, and by $(b)$ we have $x,y<_Q z$. 
Using the cancellation law (Proposition \ref{qbe-20-10}$(1)$), we get $x=y$. \\
$(c)\Rightarrow (a)$ Assume that $x\nleq_Q y$ and suppose $y\nless x$, that is $y\ra x\ne 1$. 
Hence, by Proposition \ref{qbe-60}$(1)$, we have $1\neq y\ra x=y\ra (x\Cap y)$. 
By $(c)$, we get $x=x\Cap y$, that is $x\le_Q y$, a contradiction. Thus $y< x$. 
\end{proof}

\begin{remark} \label{qlqbe-30} If $X$ is commutative, the notions of linearly ordered, weakly linear and 
quasi-linear QW algebras coincide. 
\end{remark}

\begin{example} \label{qlqbe-40} 
Let $X=\{0,a,b,c,1\}$ and let $(X,\ra,0,1)$ be the involutive BE algebra with $\ra$ and the corresponding 
operation $\Cap$ given in the following tables:  
\[
\begin{array}{c|ccccccc}
\ra & 0 & a & b & c & 1 \\ \hline
0   & 1 & 1 & 1 & 1 & 1 \\ 
a   & b & 1 & a & 1 & 1 \\ 
b   & a & 1 & 1 & 1 & 1 \\ 
c   & c & 1 & 1 & 1 & 1 \\
1   & 0 & a & b & c & 1  
\end{array}
\hspace{10mm}
\begin{array}{c|ccccccc}
\Cap & 0 & a & b & c & 1 \\ \hline
0    & 0 & 0 & 0 & 0 & 0 \\ 
a    & 0 & a & b & c & a \\ 
b    & 0 & b & b & c & b \\ 
c    & 0 & a & b & c & c \\
1    & 0 & a & b & c & 1 
\end{array}
.
\]
Then $X$ is a quantum-Wajsberg algebra, and $(X,\le_Q)$ is not linearly ordered: $a\Cap c=c\ne a$ and $c\Cap a=a\ne c$, 
hence $a\nleq_Q c$ and $c\nleq_Q a$. 
We can see that, for all $x,y\in X$ we have $x\ra y=1$ or/and $y\ra x=1$, hence $X$ is weakly linear. 
Moreover, we can easily check that, for all $x,y\in X$ such that $x\nleq_Q y$ (that is $x\ne x\Cap y$), we have 
$y\ra x=1$, that is $y< x$. It follows that $X$ is also quasi-linear. 
\end{example}

$\vspace*{5mm}$

\section{Concluding remarks}

As we mentioned, the ideals in QMV algebras have been introduced and studied  by R. Giuntini and 
S. Pulmannov\'a in \cite{Giunt7}, and these notions were also investigated in \cite{DvPu, DvPu1}.
In this paper, we defined the notions of q-deductive systems, deductive systems, maximal and strongly maximal 
q-deductive systems in QW algebras, as well as the congruences induced by deductive systems. 
We also defined the quotient QW algebra with respect to a deductive system, and investigated properties of 
quotient QW algebras depending on certain types of deductive systems. \\
R. Giuntini introduced in \cite{Giunt4} the notion of a commutative center of QMV algebras, proving that this 
structure is an MV algebra. Similarly, we can define the commutative center of a QW algebra, prove that it is a Wajsberg algebra, and  investigate its properties. \\ 
As another direction of research, based on the prime deductive systems, one can endow a QW algebra with 
certain topologies and study the properties of the corresponding topological spaces.

$\vspace*{2mm}$
          
\begin{center}
\sc Acknowledgement 
\end{center}
The author is very grateful to the anonymous referees for their useful remarks and suggestions on the subject that helped improving the presentation.

$\vspace*{5mm}$

\end{document}